\documentclass[11pt,twoside, a4paper]{article}
\usepackage[utf8]{inputenc}
\usepackage[T1]{fontenc}
\usepackage{enumitem}
\usepackage{realhats}
\usepackage{amsmath}
\usepackage{amsfonts}
\usepackage{amssymb}
\usepackage{tikz}
\usepackage{anyfontsize}
\usepackage{graphicx}
\usepackage{rotating}
\usepackage{amsthm}
\usepackage{yfonts}
\usepackage{mathtools}
\usepackage{pdfpages}
\usepackage{physics}
\usepackage{float}
\usepackage[title]{appendix}
\usepackage[
backend=biber,
style=numeric,
sorting=ynt
]{biblatex}
\usepackage{hyperref}
\hypersetup{colorlinks=true, linkcolor=red, urlcolor=blue}
\usepackage{cleveref}
\usepackage{geometry}
\usepackage{lastpage}
\usepackage{fancyhdr}
\usepackage{parskip}

\newcommand{\N}{\mathbb{N}}
 
\newcommand{\R}{\mathbb{R}}
\newcommand{\Z}{\mathbb{Z}}

\renewcommand{\epsilon}{\varepsilon}

\newtheorem{thm}{Theorem}
\newtheorem{prop}[thm]{Proposition}
\newtheorem{corollaire}[thm]{Corollary}
\newtheorem{lemma}[thm]{Lemma}
\theoremstyle{definition}

\newtheorem{defn}[thm]{Definition}
\newtheorem{rem}[thm]{Remark}
\newtheorem{expl}[thm]{Example}

\newtheorem*{notation}{Notation}

\title{Upper bounds for Steklov eigenvalues of subgraphs of polynomial growth Cayley graphs}

\author{Léonard Tschanz}

\date{}

\begin{document}
\maketitle

\begin{abstract}
We study the Steklov problem on a subgraph with boundary $(\Omega,B)$ of a polynomial growth Cayley graph $\Gamma$. We prove that for each $k \in \N$, the $k^{\mbox{th}}$ eigenvalue tends to $0$ proportionally to $1/|B|^{\frac{1}{d-1}}$, where $d$ represents the growth rate of $\Gamma$. The method consists in associating a manifold $M$ to $\Gamma$ 
and a bounded domain $N \subset M$ to a subgraph $(\Omega, B)$ of $\Gamma$. We find upper bounds for the Steklov spectrum of $N$ and transfer these bounds to $(\Omega, B)$ by discretizing $N$ and using comparison Theorems. 
\end{abstract}

\section{Introduction}

Given a smooth compact orientable Riemannian manifold $M$ of dimension $d \geq 2$ with a smooth boundary $\partial M$, the Steklov problem on $M$ is to find all $\sigma \in \R$ such that there exists a non trivial function $u$ satisfying 
\begin{align*}
\left\{
\begin{array}{lcc}
\Delta u = 0 & \mbox{on} & M \\
\frac{\partial u}{\partial \nu} = \sigma u & \mbox{on} & \partial M,
\end{array}
\right.
\end{align*}

where $\Delta$ is the Laplace-Beltrami operator acting on functions on $M$, and $\frac{\partial}{\partial \nu}$ is the outward normal derivative along $\partial M$. It is well known that the Steklov spectrum is discrete and forms a sequence such as
\begin{align*}
0 = \sigma_0 < \sigma_1 \leq \sigma_2 \leq \ldots \nearrow \infty.
\end{align*}

Among the interesting questions in the study of the eigenvalues of the Steklov problem lies the one consisting on wondering which domain of the Euclidean space maximizes the eigenvalues, and more generally, to get upper bounds, using certain assumptions such as a predefined volume. In \cite{Weinstock54}, Weinstock proves that for simply connected planar domains with an analytic boundary and assigned perimeter, the disk maximizes the first Steklov eigenvalue. In \cite{Brock01}, Brock proves that for smooth bounded domains in $\R^d$ with prescribed volume, $\sigma_1$ is maximized by the ball. Among other thing, in \cite{CEG1} upper bounds for all the eigenvalues of domains of the Euclidean space were given.

\medskip

As it is already done for the Laplace operator, one can define a discrete Steklov problem, which is a problem defined on graphs with boundary and similar to the Steklov problem defined above. This problem has recently been investigated by various authors, such as \cite{HHW, HH, Per1, Per2, CGR}. This paper will focus on the discrete Steklov problem ; let us begin by describing it briefly.


\begin{defn}
A graph with boundary is a couple $(\Gamma, B)$, where $\Gamma =(V, E)$ is a simple (that is without loop or multiple edge) connected undirected graph, and $B \subset V$ is a non empty subset of vertices called the boundary. 
The subset of vertices $B^c$ is called the interior of $\Gamma$.
\end{defn}

For $i, j \in V$, we write $i \sim j$ when $i$ is adjacent to $j$, meaning that $\{i, j\} \in E$.
For $\Omega \subset V$, we denote $|\Omega|$ the cardinal of $\Omega$, which is the number of vertices contained by $\Omega$.
In this paper, all graphs with boundary are finite. The space of real functions $u$ defined on the vertices $V$ is denoted by $\R^V$, it is the Euclidean space of dimension $|V|$. If a real function $u$ is defined only on the boundary $B$, we will say that $u \in \R^B$, and it corresponds to the Euclidean space of dimension $|B|$. For $u, v \in \R^V$, the scalar product $\langle u, v \rangle$ is the usual scalar product in $\R^{|V|}$.  We now introduce the Laplacian operator $\Delta :\R^V  \longrightarrow  \R^V $, defined by
\begin{align*}
\Delta u :  V & \longrightarrow \R \\
            i & \longmapsto \Delta u(i) = \sum_{j \sim i}(u(i)-u(j)),
\end{align*}
as well as the normal derivative $\frac{\partial}{\partial \nu} :\R^V  \longrightarrow \R^B $, defined by
\begin{align*}
\frac{\partial u}{\partial \nu} : B & \longrightarrow  \R \\ 
                                i & \longmapsto \frac{\partial u}{\partial \nu}(i) = \sum_{j \sim i} (u(i)-u(j)).
\end{align*}

\begin{defn}
The Steklov problem on a graph with boundary $(\Gamma, B)$ consists in finding all $\sigma \in \R$ such that there exists a non trivial function $u \in \R^V$ satisfying 
\begin{align*}
\left\{
\begin{array}{lcc}
\Delta u(i) = 0 & \mbox{if} & i \in B^c \\
\frac{\partial u}{\partial \nu}(i) = \sigma u(i) & \mbox{if} & i  \in B.
\end{array}
\right.
\end{align*}
Such a $\sigma$ is called a Steklov eigenvalue of $(\Gamma, B)$.
\end{defn}
As explained in \cite{Per1}, the Steklov spectrum forms a sequence as follows
\begin{align*}
0 = \sigma_0 < \sigma_1 \leq \sigma_2 \leq \ldots \leq \sigma_{|B|-1}.
\end{align*}
Recent interesting outcomes related to this problem include the following :

In \cite{HHW}, the authors find a Cheeger type inequality for the first non trivial eigenvalue. In \cite{Per1}, the author finds a lower bound for the first non trivial Steklov eigenvalue. Lower bounds for higher Steklov eigenvalues are given in \cite{HM}.

In \cite{CGR},the authors described a process, called a discretization, permitting to associate a graph with boundary to a Riemannian manifold and showed some spectral bonds between a manifold and its discretization. 
\medskip

\begin{defn}
Given a graph $\Gamma = (V,E)$ and a finite connected subset $\Omega \subset V$, one can define a graph with boundary $(\bar{\Omega}, E', B)$ included in $\Gamma$, by saying
\begin{itemize}
\item $ B = \{j  \in V \backslash \Omega : \exists \; i \in \Omega  \; \mbox{such that} \; \{i,j\} \in E\}$ ;
\item $\bar{\Omega} = \Omega \cup B$ ;
\item $E' = \{ \{i,j\} \in E : i \in \Omega  , j \in \bar{\Omega} \}$.
\end{itemize}
Such a graph is called a subgraph of $\Gamma$ and is denoted $(\Omega, B)$, with $\Omega$ the interior and $B$ the boundary.
\end{defn}

Following the work of Brock cited above, Han and Hua found in \cite{HH} a similar result for subgraphs of the integer lattices :
\begin{thm}[Han, Hua, 2019] \label{corollaire : sigma_1 HH}
Let $\Z^d$ be the integer lattice of dimension $d$. Let $(\Omega, B)$ be a subgraph of $\Z^d$. Then we have
\begin{align*}
\sum_{l=1}^{d} \frac{1}{\sigma_l(\Omega, B)} \geq \bar{C} \cdot |\Omega|^{\frac{1}{d}} - \frac{C'}{|\Omega|},
\end{align*}
where $\bar{C}= (64 d^3 \omega_d^{\frac{1}{d}})^{-1}, C'=\frac{1}{32d}$ and $\omega_d$ is the volume of the unit ball in $\R^d$. 
\end{thm}
This Theorem gives us control over the $d$ first Steklov eigenvalues of a subgraph $(\Omega, B)$ of $\Z^d$ and leads to an interesting consequence : for any sequence of subgraphs $(\Omega_l, B_l)_{l=1}^\infty$ of $\Gamma$ such that $|\Omega_l| \underset{l\to \infty}{\longrightarrow} \infty$, we have $\sigma_1( \Omega_l, B_l) \underset{l\to \infty}{\longrightarrow}  0$.
\medskip

However, unlike Brock, Han and Hua do not get an equality case.
\medskip

The result of Han and Hua has then been partially extended by Perrin, who found an isoperimetric upper bound for the first eigenvalue of subgraphs of any polynomial growth Cayley graph, see \cite{Per2}.
We will recall in section \ref{section : 2} what a polynomial growth Cayley graph is and the other notions of geometric group theory which are necessary for the understanding of the paper.

\begin{thm}[Perrin, 2020] \label{corollaire : sigma_1 HP}
Let $\Gamma = (V,E)$ be a Cayley graph with polynomial growth of order $d \geq 2$. There exists $\tilde{C}(\Gamma) >0$ such that for any finite subgraph $(\Omega, B)$ of $\Gamma$, we have
\begin{align*}
\sigma_1(\Omega, B) \leq \tilde{C}(\Gamma) \cdot \frac{1}{|B|^{\frac{1}{d-1}}}.
\end{align*}
\end{thm}

This Theorem gives us control over the first Steklov eigenvalue of a subgraph of any polynomial growth Cayley graph and leads to the same consequence as the previous Theorem.
\medskip

It is therefore natural to wonder whether it is possible to extend our control to all of the eigenvalues of a subgraph of any polynomial growth Cayley graph.
\medskip

The main result of this paper is the following :
\begin{thm} \label{thm : principal}
Let $\Gamma = Cay(G,S)$ be a polynomial growth Cayley graph of order $d \geq 2$. Let $(\Omega, B)$ be a subgraph of  $\Gamma$. Then there exists a constant $C(\Gamma) >0$ such that for all $k < |B|$,
\begin{align*}
\sigma_k(\Omega, B) \leq C(\Gamma) \cdot \frac{1}{|B|^{\frac{1}{d-1}}} \cdot k^{\frac{d+2}{d}}.
\end{align*}
\end{thm}
One can observe that the bound depends on the cardinal of the boundary in the same way as Perrin's one in Theorem \ref{corollaire : sigma_1 HP}.

This leads to a consequence that extends the one of Theorems \ref{corollaire : sigma_1 HH} and \ref{corollaire : sigma_1 HP} :

\begin{corollaire} \label{corollaire : sigma_k}
Let $\Gamma$ be a polynomial growth Cayley graph of order $d \geq 2$ and  $(\Omega_l, B_l)_{l=1}^\infty$ be a sequence of subgraphs of $\Gamma$ such that $|\Omega_l| \underset{l \to \infty}{\longrightarrow} \infty$. Fix $k \in \N$. Then we have
\begin{align*}
\sigma_k(\Omega_l, B_l) \underset{l \to \infty}{\longrightarrow} 0.
\end{align*}
\end{corollaire}

Of course, the number $\sigma_k(\Omega_l, B_l)$ is defined as long as $k < |B_l|$, which is the case for an $l$ big enough due to the assumption that $|\Omega_l| \underset{l \to \infty }{\longrightarrow}\infty$, see Proposition \ref{prop : isoperimetrie} for more details. 
\medskip

Our approach to proving this is completely different from \cite{Per2} and looks more like \cite{HH}, in the sense that we do not work directly on graphs. However, the detour we make is not the same as the one made by \cite{HH}. We use tools described by Colbois, Girouard and Raveendran in \cite{CGR} to build a manifold associated to $\Gamma$ and a bounded domain associated to $(\Omega, B)$.  We then use results from Colbois,  El Soufi and Girouard \cite{CEG1} to give upper bounds for the Steklov eigenvalues of the domain, then apply Theorems of \cite{CGR} to transfer these upper bounds to the subgraph $(\Omega, B)$ by discretizing the domain into a graph with boundary that corresponds to $(\Omega, B)$.

\begin{notation}
Throughout the paper, we shall work on graphs and on manifolds. Graphs are denoted $\Gamma = (V, E)$ and manifolds are denoted $M$. The couple $(M,g)$ means that $M$ is endowed with a Riemannian metric $g$ and we use $|\cdot|_g$ to denote the Riemannian volume of a subset of $M$, as well as $d_M$ to denote the distance on $M$. We denote by $(N,\Sigma)$ a bounded domain of $M$, with $N$ the interior and $\Sigma$ the boundary. We shall use the variables $e, i, j, v$ to speak about vertices of graphs and $x, y, z$ for points on manifolds. Several constants will appear, we shall call them $C_1, C_2, \ldots$ ; each $C_l$ is used exactly once.
\end{notation}

\textbf{Plan of the paper.} In section \ref{section : 2} we recall some definitions and results about geometric group theory that are needed for the constructions that will follow. In section \ref{section : 3} we build a manifold $M$ that is modeled on a Cayley graph $\Gamma$ and prove some Propositions that will allow us to use results that we need on $M$. In section \ref{section : 4} we prove Theorem \ref{thm : principal} : we first explain how to associate a bounded domain $N \subset M$ to a subgraph $(\Omega, B)$ of   $\Gamma$ and we use it to obtain an upper bound for the Steklov eigenvalues of $(\Omega, B)$, which will allow us to conclude.
\medskip

\textbf{Acknowledgments.} I would like to warmly thank my thesis supervisor Bruno Colbois for having offered to work on this subject as well as for his many advice which enabled me to resolve the difficulties encountered. I also wishes to thank Niel Smith and Antoine Gagnebin for their careful rereading of this paper and for their various remarks which have led to its improvement.

\section{Cayley graphs and isoperimetric inequality} \label{section : 2}
Because we will work with graphs with boundary that are subgraphs of a polynomial growth Cayley graph, we recall here some definitions that we will use, as well as some properties satisfied by these graphs.
For further details, see \cite{DLH}.
\medskip

Let $G$ be a finitely generated infinite discrete group and $S \subset G$ be a symmetric ($S = S^{-1}$), finite generating subset such that $e \notin S$, where $e$ denote the identity of $G$. The Cayley graph associated to $(G,S)$ is an infinite connected undirected simple graph $Cay(G,S) =(V,E)$ of vertices $V = G$ endowed with the graph structure $E = \{\{i,j\} : i,j \in V  \; : \; \exists \; s \in S \; \mbox{such that} \; j = is\}$.

\begin{rem}
Hence a Cayley graph is regular, each vertex $i$ has degree $d(i) = |S|$. 
\end{rem}

A graph is a metric space when endowed with the path distance.  A path is a sequence of vertices $i_1 \sim i_2 \sim \ldots \sim i_{n+1}$.  If $i = i_1 \sim i_2 \sim \ldots \sim i_{n+1} =j$ is a minimal path joining $i$ to $j$, then we say that the distance between $i$ and $j$ is $n$. In particular, for $i, j \in Cay(G,S)$ such that $i \sim j$, we have $d(i,j) =1$.

We denote by $B(n)$ the ball of radius $n$ centered at $e \in Cay(G,S)$ . The growth function of $Cay(G,S)$ is defined by $V(n) = |B(n)|$. If there exists $d \in \N$ and $C_1 \geq 1$ such that for all $n \in \N$, we have
\begin{align*}
C_1^{-1} \cdot n^d \leq V(n) \leq C_1 \cdot n^d,
\end{align*}
we say that $Cay(G,S)$ is a Cayley graph with a growth rate that is polynomial of order $d$. It is well known that the growth rate does not depend on the choice of the subset $S$ (\cite{DLH}, chapter VI). Hence we can speak about the growth rate of the group $G$.

\begin{expl}
Let $G = \Z^d$ and $S = \{(\pm 1, 0, \ldots , 0), (0, \pm 1, 0, \ldots, 0), \ldots, (0, \ldots, 0, \pm 1) \}$.  Then $Cay(G,S)$ is a Cayley graph with polynomial growth rate of order $d$, called the integer lattice of dimension $d$ and simply denoted by $\Z^d$.
\end{expl}

\begin{prop} \label{prop : isoperimetrie}
Let $G$ be a group with polynomial growth rate of order $d$, $S$ a finite symmetric generating set, $e \notin S$ and $Cay(G,S)$ as above. Let $(\Omega, B)$ be a subgraph of  $Cay(G, S)$. Then there exists a constant $C_2$ depending only on $G$ and $S$ such that
\begin{align*}
\frac{|\bar{\Omega}|^{\frac{d-1}{d}}}{|B|} \leq C_2.
\end{align*}
\end{prop}
This isoperimetric control will be very useful to conclude our proof of Theorem \ref{thm : principal}. For a proof of this proposition, one can see Theorem $1$ and its  illustration in the example below in \cite{CS}.

From here, we assume $d \geq 2$. This will guarantee that if $|\bar{\Omega}|$ is big, then so is $|B|$.


\section{Manifold modeled on graphs} \label{section : 3}

In order to use the results presented in \cite{CEG1}, we have to work on manifolds. This section is devoted to explain how we can associate a manifold $M$ to a Cayley graph $\Gamma$, in such a way that a subgraph $(\Omega, B)$ of $\Gamma$ corresponds to a discretization of a bounded domain $N$ of $M$. The idea comes from the work of Colbois, Girouard and Raveendran, see \cite{CGR}, where they construct manifolds with some desired properties that these manifolds share with their discretizations. 

We shall now explain how to construct a manifold that is modeled on a Cayley graph.
\medskip

Let $\Gamma = (V,E) = Cay(G,S)$ be a Cayley graph. We build what we call a \textit{fundamental piece} $(P, g_0)$, that is a smooth compact $d$-dimensional Riemannian manifold  with $|S|$ boundary component, homeomorphic to $\mathbb{S}^{d}$ with $|S|$ holes. Each boundary component possesses a neighborhood which is isometric to the cylinder $[0,2] \times \mathbb{S}^{d-1}$, with the boundary corresponding to $\{0\} \times \mathbb{S}^{d-1}$, as seen in figure \ref{fig : piece fondamentale simple}. On the cylindrical neighborhood of the boundary, $g_0$ is expressed as a product metric.

\begin{figure}[H]
\centering
\includegraphics[scale=0.17]{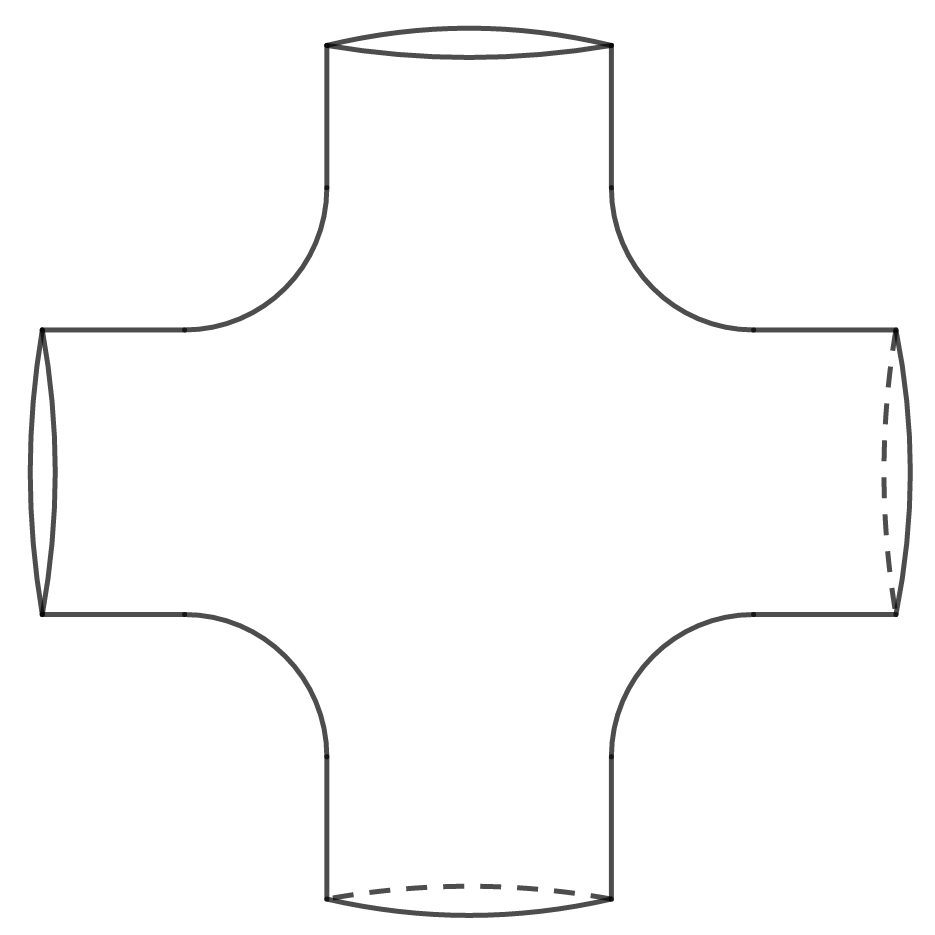}
\caption{A fundamental piece associated with the lattice $\Z^2$.}
\label{fig : piece fondamentale simple}
\end{figure}

\begin{rem} \label{rem : piece pas bien definie}
Outside the $2$-neighborhood of the boundary, we do not specify the geometry of $P$. The only thing we impose is that the piece is smooth. In subsection \ref{subsection : domaine associe}, we will particularize the geometry of our piece, but since this specification in not yet relevant, we do not emphasize it now.
\end{rem}

From this fundamental piece $P$, we construct a smooth unbounded complete $d$-dimensional  Riemannian manifold $(M,g)$, gluing infinitely many copies of $P$. For each vertex $i \in V$ we add one copy of $P$, denoted $P_i$. We call $P_i$ a piece of $M$. It is obvious that these pieces can be glued smoothly along their boundary because of the boundary's cylindrical neighborhood, see figure \ref{fig : variete modele}. The metric $g$ comes from the metric $g_0$ and $g$ is smooth because for each gluing part of $M$ there is a neighborhood isometric to the cylinder $[-2,2] \times \mathbb{S}^{d-1}$, with the gluing part corresponding to $\{0\} \times \mathbb{S}^{d-1}$ and where $g$ can be expressed as a product metric. As said before, on $[-2,0] \times \mathbb{S}^{d-1}$ and on $[0,2] \times
\mathbb{S}^{d-1}$, $g_0$ is a product metric. Hence, $g$ is smooth.

\begin{rem}
We do not specify which diffeomorphism is used to glue the pieces together, hence the manifold $M$ is not entirely well defined. This is not a problem for us, the properties we need about $M$ and that we shall prove are verified by any element of the family of manifolds described by our process. We pick one and call it $M$ for the purpose of this paper.
\end{rem}

Because $|V| = \infty$, $M$ is unbounded. Because the number of boundary components of $P$ is exactly $|S|$, this construction leads to a correspondence between $\Gamma$ and $M$, in such a way that if $i \sim j$, then $P_i$ is glued to $P_j$. Now that we have emphasized the links between $M$ and $\Gamma$ we can call $M$ a manifold \textit{modeled} on $\Gamma$. See \cite{CGR} for more applications of these manifolds modeled on graphs.

\medskip

From now we assume that $\Gamma$, as above, is a Cayley graph with polynomial growth of order $d \geq 2$ with a growth rate constant  $C_1$.

We now give some properties that are satisfied by such a manifold $M$ modeled on a Cayley graph $\Gamma$.

We define 
\begin{align*}
\varphi : \bigcup_{i \in V} \mbox{int}P_i & \longrightarrow V \\
             x & \longmapsto v_x
\end{align*}
where $v_x$ is the vertex of $\Gamma$ associated with the piece of $M$ to which $x$ belongs.
\begin{rem}
$\varphi$ is well defined because $x$ does not belong to the boundary of a piece of $M$. We can then extend $\varphi$ to the whole manifold : for $x \in M$ that belongs to the boundary of a piece, one of the two possibilities is chosen  once and for all. This extended map is called $\varphi$ again.
\end{rem}

\begin{rem}
As explained before, $\Gamma$ is endowed with the path distance, denoted $d_\Gamma$. By construction of $M$, if $x$ and $y$ do not belong to the same piece, then $d_\Gamma(v_x, v_y)$ represents the number of pieces that  must be crossed to go from $x$ to $y$ plus one.
\end{rem}

\begin{figure}[H]
\centering
\includegraphics[scale=0.1]{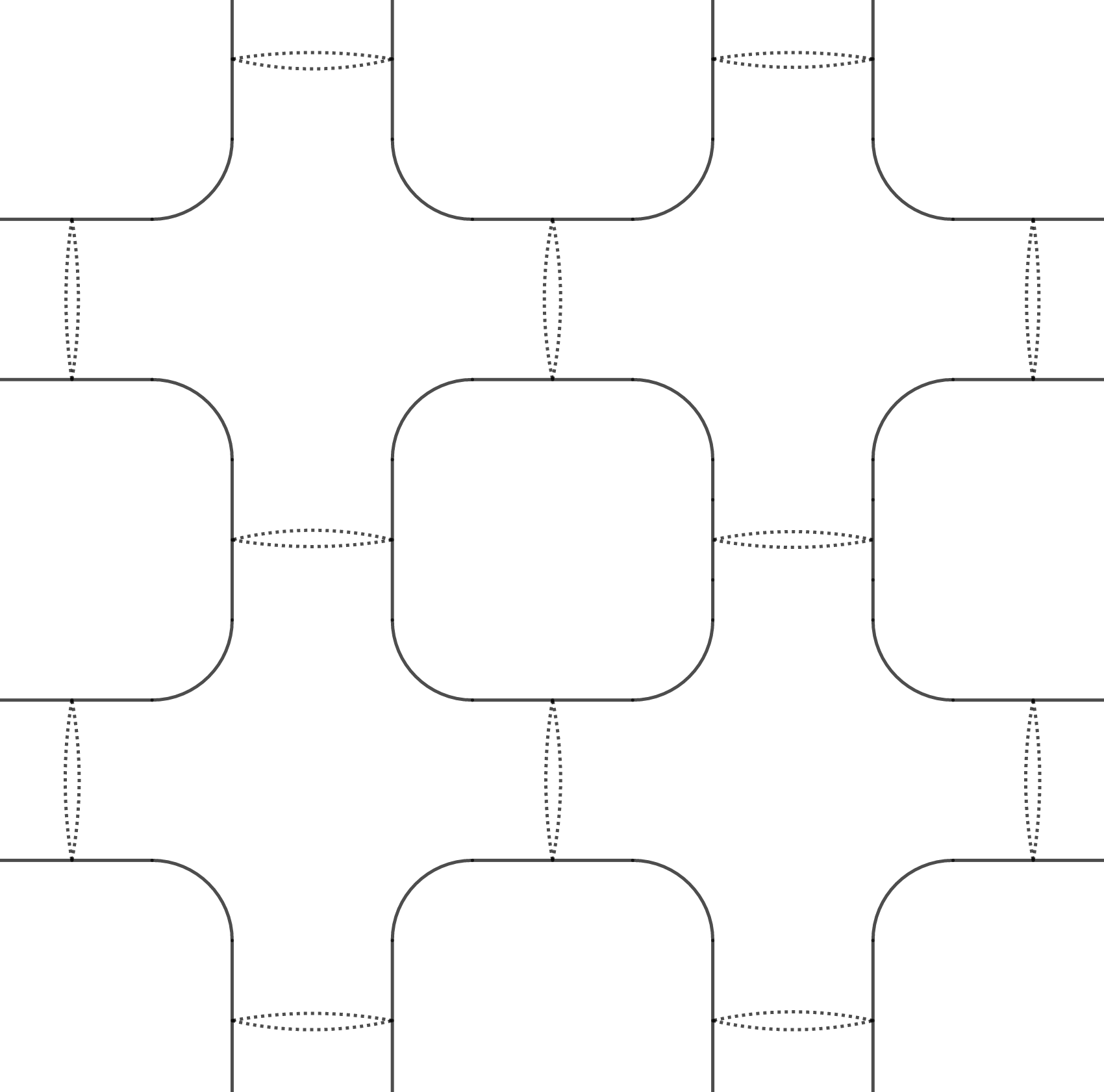}
\caption{Example of a manifold modeled on the lattice $\Z^2$.}
\label{fig : variete modele}
\end{figure}  

We now prove some results about $(M,g)$.
\begin{lemma} \label{lemme : distance}
There exist constants $C_3, C_4 >0$, depending only on $P$, such that for all $x, y \in M$ with $d_M(x,y)  \geq C_3$, we have 
\begin{align*}
C_4^{-1} \cdot d_\Gamma(v_x, v_y) \leq d_M(x,y) \leq C_4 \cdot d_\Gamma(v_x, v_y).
\end{align*}
\end{lemma}

\begin{proof}
Let $C_3$ be the diameter of $P$ plus one, that is $C_3 = \mbox{diam}P+1 = \sup \{d_M(x,y) : x,y \in P\} +1$. Then for $x,y \in M$ such that $d_M(x,y) \geq C_3$, $x, y$ cannot belong to the same piece of $M$. Let $C_4 := 2 \cdot \mbox{diam}P+1$. Remember that the number $d_\Gamma(v_x, v_y)$ represents the number of pieces that  must be crossed to go from $x$ to $y$ plus one.
Then for $x,y \in M$ such that $d_M(x,y) > C_3$, we have
\begin{align*}
C_4^{-1} \cdot d_\Gamma(v_x, v_y) \leq d_M(x,y) \leq C_4 \cdot d_\Gamma(v_x, v_y).
\end{align*}
\end{proof}

\begin{lemma}
There exist constant $C_5, C_6 >0$, depending only on $\Gamma$ and $P$, such that for all $x \in M$ and $r > C_3$, we have 
\begin{align*}
C_5 \cdot r^d \leq |B(x, r)|_g \leq C_6 \cdot r^d.
\end{align*}
\end{lemma}

\begin{proof}
We use Lemma \ref{lemme : distance} to compare the distance in $\Gamma$ with the one in $M$. First we prove the right hand inequality.
\begin{align*}
|B(x,r)|_g & \stackrel{def}{=} |\{y  \in M : d_M(x,y) < r\}|_g \\
           & \leq |P|_g \cdot |\{v_y : y \in B(x,r)\}| \\
           & \leq |P|_g \cdot |B(v_x, r)| \cdot C_4 \\
           & \leq |P|_g \cdot C_1 \cdot  r^d \cdot C_4 \\
           & =: C_6 \cdot r^d.
\end{align*}
Now we prove the left hand side inequality.
\begin{align*}
|B(x,r)|_g & \stackrel{def}{=} |\{y  \in M : d_M(x,y) < r\}|_g \\
           & \geq |P|_g \cdot |\{v_y : y \in P_{g_y} \subset B(x,r) \}| \\
           & \geq |P|_g \cdot |B(v_x, r)| \cdot C_4^{-1} \\
           & \geq |P|_g \cdot C_1^{-1} \cdot r^d \cdot C_4^{-1} \\
           & =: C_5 \cdot r^d.
\end{align*}
\end{proof}

We just showed that for all $x \in M$ and for $r$ big enough, the volume of the ball centered at $x$ with radius $r$ is proportional to $r^d$. But as $M$ is a $d$-dimensional manifold, which means that $M$ is locally homeomorphic to $\R^d$,  this is also true for $r$ small enough. What we mean is, given $x \in M$, we can choose $r_0 >0$ sufficiently small to find $C_{r_0}, C_{r_0}'$ such that for all $r \leq r_0$ we have $C_{r_0} \cdot r^d \leq  |B(x,r)|_g \leq C_{r_0}' \cdot r^d$. 

For $x \in P$, call $r_0(x) \in (0, C_3]$ the bigger number such as above and define $r^* = \inf \limits_{x \in P}r_0(x)$. Because $P$ is smooth, the function $x \longmapsto r_0(x)$ is continuous. Moreover, $P$ is compact. Hence the number $r^*$ is strictly positive. Now recall that $M$ is obtained by gluing copies of this unique fundamental piece and we can conclude that $r^*$ is a uniform bound valid for any $x \in M$. We call $C_*, C_*'$ the constant satisfying $C_*  \cdot r^d \le |B(x,r)|_g \le C_*'  \cdot r^d$ for all $x \in M$ and $r \le r^*$.

So for all $x \in M$, we have that $|B(x,r)|_g \approx r^d$ for $r$ big enough and small enough. This leads to
\begin{prop} \label{prop : croissance poly}
There exist constants $C_7, C_8 >0$, depending only on $\Gamma$ and $P$, such that  for all $x \in M$ and $r >0$, we have
\begin{align*}
C_7 \cdot r^d \leq |B(x,r)|_g \leq C_8 \cdot r^d.
\end{align*}
\end{prop}

\begin{proof}
We already know this is true for all $r \leq r^*$ and for all $r > C_3$. Now consider $C_9 :=   \inf \limits_{x \in P} |B(x, r^*)|_g $ and $C_{10} :=  \sup \limits_{x \in P} |B(x,C_3)|_g$, which are strictly positive finite numbers following the same arguments as above. Then for all $x \in M$ and all $r^*  \le r \le C_3$, we have $ C_9 \le |B(x,r)|_g \le C_{10}$. Define $C_7 := C_* \cdot C_9$ and $C_8 := C_6 \cdot C_{10}$ and we are done.
\end{proof}

The next proposition is a packing property :
\begin{prop} \label{prop : packing}
There exists a constant $C_{11} \geq 1$, depending only on $\Gamma$ and $P$ such that for all $r >0$ each ball of radius $2r$ in $M$ can be covered by $C_{11}$ balls of radius $r$.
\end{prop}

\begin{proof}
Let $x \in M$ and $B(x,2r)$ be the ball centered at $x$ with radius $2r$. Choose a maximal set of $C_{11}$ points $x_i \in B(x, \frac{3r}{2})$ such that $d_M(x_i, x_j) \geq  \frac{r}{2}$ for $i \neq j$. It is clear by construction that the balls $B(x_i, \frac{r}{4})$ are mutually disjointed. Moreover, the balls $B(x_i, r)$ cover $B(x,2r)$. In order to show this, let $y \in B(x, 2r)$. Then there exists $y' \in B(x, \frac{3r}{2})$ such that $d_M(y,y') \leq \frac{r}{2}$. Because the set $\{x_i\}_{i=1}^{C_{11}}$ is maximal, there exists $1 \leq j \leq C_{11}$ such that $d_M(x_j, y') \leq \frac{r}{2}$. then by the triangular inequality,
\begin{align*}
d_M(x_j,y) \leq d_M(x_i, y') + d_M(y', y) \leq \frac{r}{2} + \frac{r}{2} = r,
\end{align*} 
which means that $y \in B(x_j, r)$. 

Then we have
\begin{align*}
C_{11} & \leq \frac{|B(x,2r)|_g}{\min \limits_{i}|B(x_i,\frac{r}{4})|_g} & &  \\
       & = \frac{|B(x,2r)|_g}{|B(x_0,\frac{r}{4})|_g}                    & &  \mbox{for a certain} \; x_0, \\
       & \leq  \frac{|B(x_0,4r)|_g}{|B(x_0,\frac{r}{4})|_g}              & & \mbox{because} \; B(x, 2r) \subset B(x_0, 4r), \\
       & \leq \frac{C_8 \cdot (4r)^d}{C_7 \cdot (\frac{r}{4})^d}         & & \\
       & = 16^d \cdot \frac{C_8}{C_7}.                                    & & \\
\end{align*}
\end{proof}

\section{Proof of the main Theorem} \label{section : 4}

We split the proof of Theorem \ref{thm : principal} into several Propositions that will lead us to the conclusion. 
\subsection{Domain associated to a subgraph} \label{subsection : domaine associe}

First we explain how to associate a bounded domain  $(N, \Sigma) \subset M$ to a subgraph $(\Omega, B)$ of $\Gamma$.

What we have to keep in mind is that we want to preserve the structure of $(\Omega, B)$ in the domain $N$. On this purpose, as said in remark \ref{rem : piece pas bien definie}, we have to specify the geometry of our fundamental piece $P$.
\medskip

The metric $g_0$ on $P$ is such that there exists a point $z \in P$ and an annulus $A(z,1, 3)= \{x \in P : 1 < d_M(z,x) <3\}$ which is isometric to the cylinder $[0,2] \times \mathbb{S}^{d-1}$. This annulus does not intersect any cylindrical neighborhood of the boundary, see figure \ref{fig : piece fondamental}.
\medskip

\begin{rem}
Because $P$ is still a smooth compact $d$-dimensional Riemannian manifold homeomorphic to $\mathbb{S}^d$ containing $|S|$ holes, with an appropriated neighborhood of the boundary, all Propositions we have stated about $M$ are verified.
\end{rem}

This annulus is the key point of our construction : if we decided to remove $B(z, 1)$ from $P$, then we would obtain a manifold which would be homeomorphic to a $d$-dimensional sphere with $|S| +1$ holes with the property that each boundary component possess a neighborhood isometric to the cylinder $[0,2] \times \mathbb{S}^{d-1}$. We will have to remove this annulus from some particular pieces that shall compose $N$ in order to guarantee the existence of a connected component of $\Sigma$ on those pieces. We will show in example \ref{expl : anneau} that this trick is necessary.

\begin{figure}[H]
\centering
\includegraphics[scale=0.17]{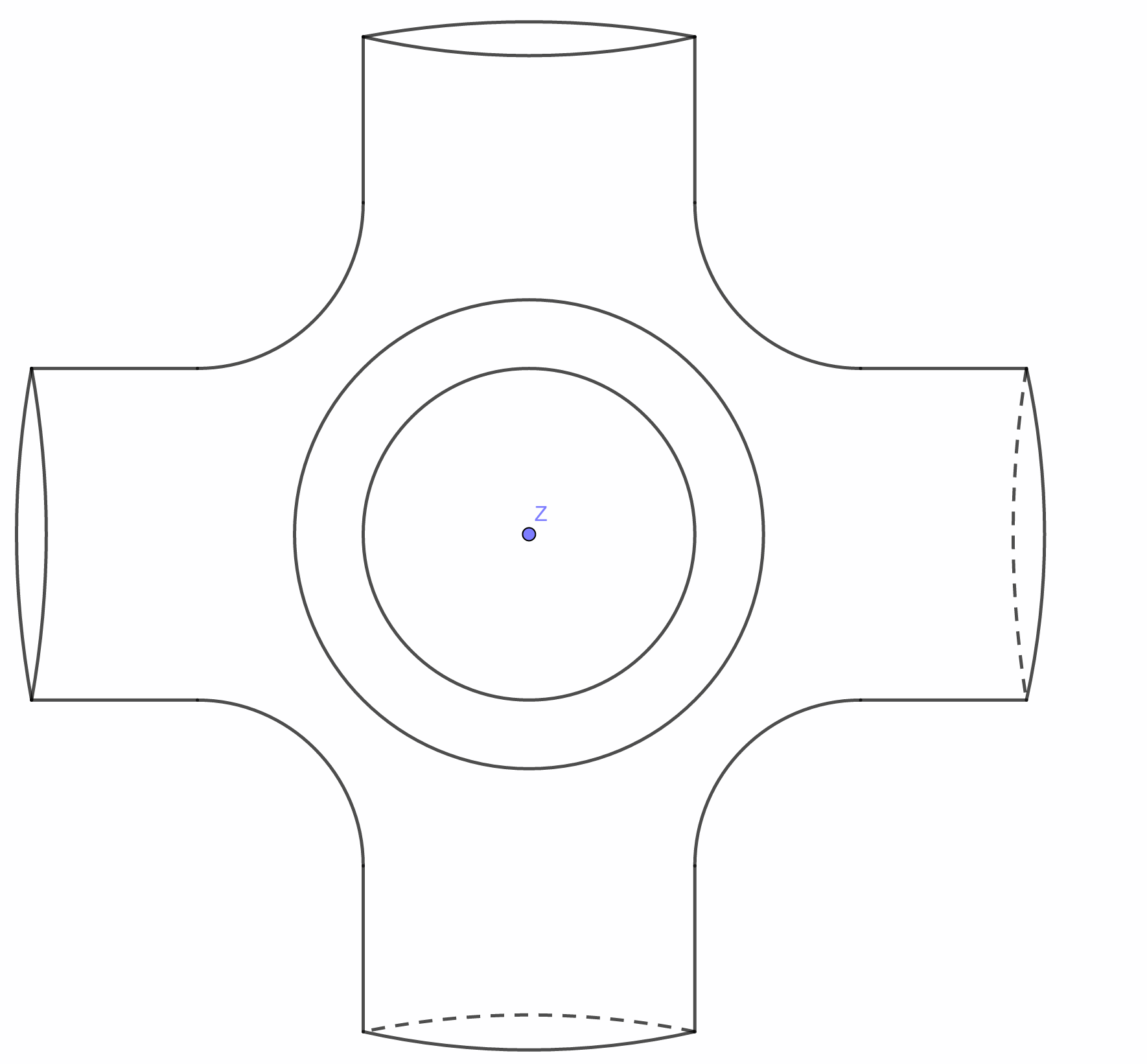}
\includegraphics[scale=0.2]{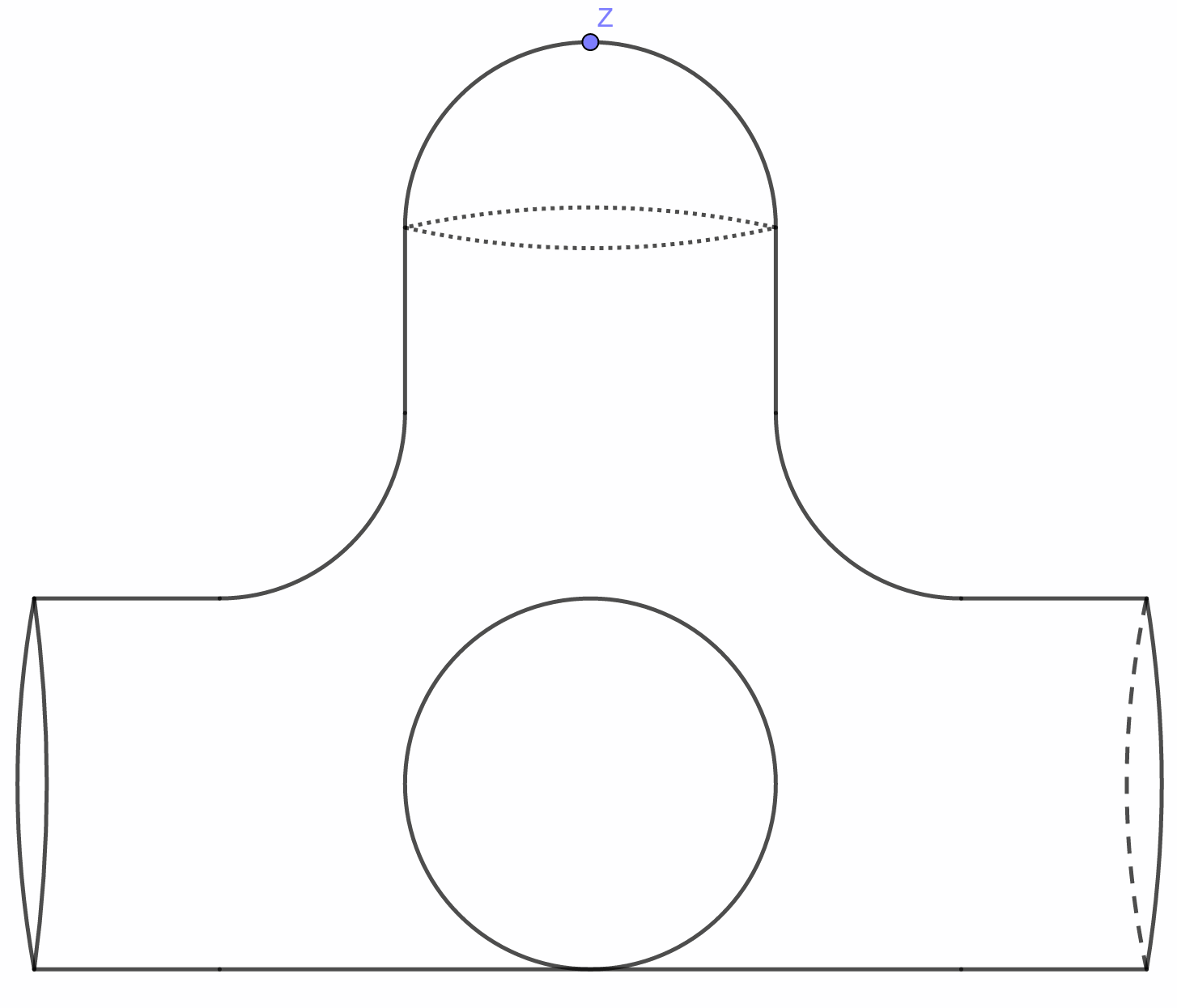}
\caption{Example of a fundamental piece associated with the lattice $\Z^2$. Removing the  ball $B(z, 1)$ leads to a fifth hole like the other four ones. The picture on the right is a view from the side.}
\label{fig : piece fondamental}
\end{figure}  

The copy of the point $z \in P$ on a piece $P_i$ associated to $i$ is denoted $z_i$.


Given a subgraph $(\Omega, B)$ of $\Gamma$, we shall associate a bounded domain $(N, \Sigma) \subset M$ to it. Proceed as follows :

\begin{itemize}
\item For each $i \in \bar{\Omega}$ we take $P_i$ the piece of $M$ associated to $i$ ;
\item If $i, j \in \bar{\Omega}$ are such that $i \sim j$ in $\Gamma$ but $i \nsim j$ in $(\Omega, B)$ (which could happen if $i, j \in B$), take the pieces $P_i$ and $P_j$ but disconnect them by removing a cylinder isometric to $[-1, 1]  \times \mathbb{S}^{d-1}$, where $\{0\} \times \mathbb{S}^{d-1}$ corresponds to the gluing part of these pieces ;
\item For all $j \in B$ remove the ball $B(z_j,1)$ from the piece $P_j$.
\end{itemize}
This gives us a natural boundary that we will call $\Sigma$, composed of several disjoint copies of $\mathbb{S}^{d-1}$, see figure \ref{fig : domaine variete modele}.

\begin{rem}
For $i \in \bar{\Omega}$, we continue to call $P_i$ the piece of $N$ that is associated to $i$, even if this piece is not a whole one.
\end{rem}

\begin{rem}  \label{rem : trick anneau}
The purpose of this maneuver is to imitate the structure of $(\Omega, B)$ on $(N, \Sigma)$. Our construction guarantees that for $i, j \in \bar{\Omega}$, we have the equivalence $i \sim j \Leftrightarrow P_i \sim P_j$, where $P_i \sim P_j$ means that $P_i$ is connected to $P_j$. Moreover, the boundary structure is preserved, that is $j \in B \Leftrightarrow P_j$ contain at least one connected component of the boundary.
The trick of the annulus is essential ; example \ref{expl : anneau} shows us that without it, the structure of the subgraph might not be reproduced by the domain.
\end{rem}

\begin{figure}[H]
\centering
\includegraphics[scale=0.06]{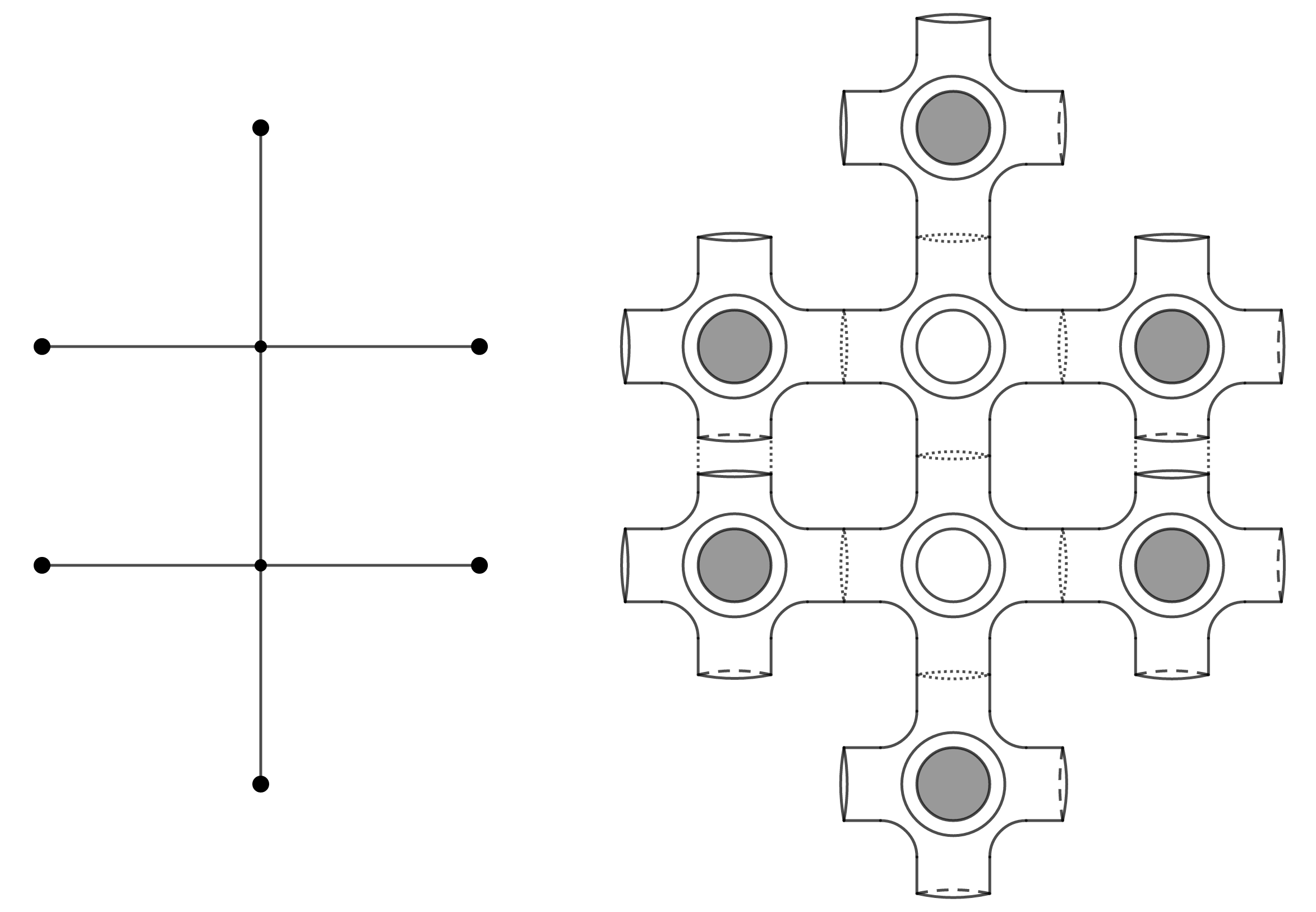}
\caption{Example of a subgraph of $\Z^2$ and a domain associated. On the left, the big dots represent the boundary $B$ while the small ones represent the interior $\Omega$. On the right, the grey balls are removed from the domain.}
\label{fig : domaine variete modele}
\end{figure}

\begin{expl} \label{expl : anneau}
Look at the lattice $\Z^2$, and let $\Omega = B(n)\backslash e$ be the centered ball of radius $n$, deprived of the origin. Then, for the induced subgraph $(\Omega, B)$, we have $e \in B$. Moreover, $e$ is adjacent to each of its four neighbors in $\Z^2$. Then the boundary of the domain would not have any component close to $P_e$ if we did not remove the ball $B(z_e, 1)$ from the piece $P_e$.
\begin{figure}[H]
\centering
\includegraphics[scale=0.06]{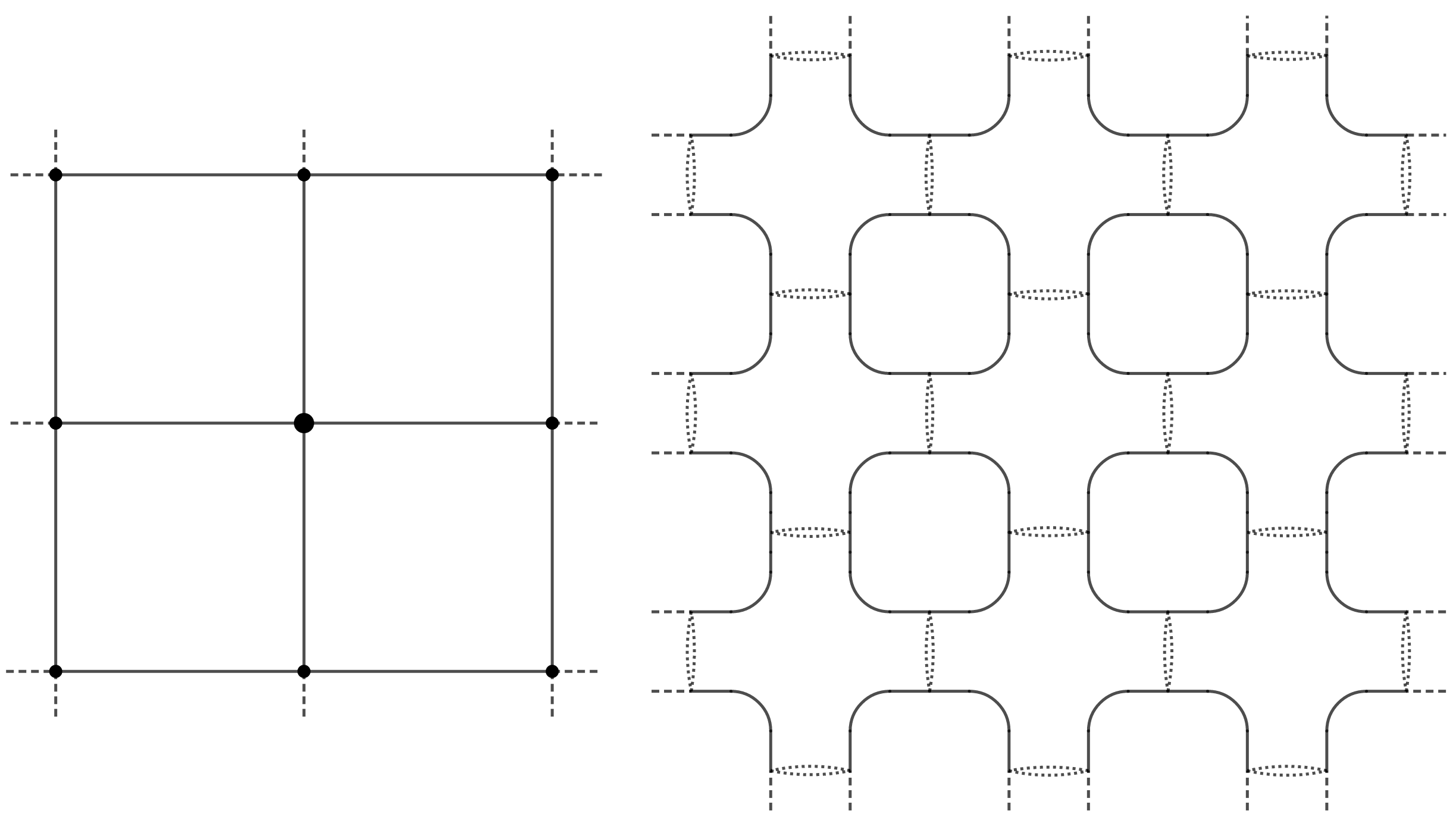}
\caption{Without the trick of the annulus, the boundary structure of the subgraph might not be reproduced on the associated domain : there is no boundary component near the central piece.}
\end{figure}
\end{expl}

Because $\Omega$ is chosen connected, so is the domain $(N, \Sigma)$. We denote by $|\Sigma|$ the $(d-1)$-volume of $\Sigma$, that is
\begin{align*}
|\Sigma| := \int_\Sigma dv_ \Sigma,
\end{align*}
where $dv_\Sigma$ is the measure induced by the Riemannian metric of $M$ restricted to $\Sigma$.
\begin{rem} \label{rem : vg isom}
$\Sigma$ possesses a neighborhood that is isometric to the cylinder $[0,1] \times \mathbb{S}^{d-1}$, for we took care of constructing a fundamental piece $P$ with a boundary that admits a neighborhood isometric to the cylinder $[0,2] \times \mathbb{S}^{d-1}$. Moreover, the annulus $A_i = A(z_i,1 ,3)$ was built in a way that removing $B(z_i, 1)$ leads to a component  of the boundary that possess a neighborhood isometric to $[0, 2] \times \mathbb{S}^{d-1}$, see figure \ref{fig : piece fondamental}.
\end{rem}
For $(N, \Sigma)$ a bounded domain in $M$, we introduce the isoperimetric ratio $I(N)$ defined by
\begin{align*}
I(N) = \frac{|\Sigma|}{|N|_g^{\frac{d-1}{d}}}.
\end{align*}

The idea is now to compare the eigenvalues of $(\Omega, B)$ with $N$'s ones. In order to do it, we state now a Theorem from Colbois, El Soufi and Girouard that gives us an upper bound for the Steklov spectrum of $N$.
\begin{thm} \label{thm : CEG}
Let $(M, g)$ be a complete $d$-dimensional manifold that satisfies properties of Proposition \ref{prop : croissance poly} and \ref{prop : packing}. Then there exists a constant $C_{12} = C_{12}(g)$ depending only on $C_{11}, C_7$ and $C_8$ coming from Propositions \ref{prop : croissance poly} and \ref{prop : packing} such that for any bounded domain $(N, \Sigma) \subset (M,g)$, we have for every $k \geq 0$,
\begin{align*}
\sigma_k(N) \cdot |\Sigma|^{\frac{1}{d-1}} \leq \frac{C_{12}}{I(N)^{\frac{d-2}{d-1}}} \cdot k^{\frac{2}{d}}.
\end{align*}
\end{thm}
For a proof of this, we refer to Theorem $2.2$ of \cite{CEG1}. Actually, the result obtained by Colbois, El Soufi and Girouard is a little bit more general than that, but this statement is enough for our need.
\begin{rem}
The constant $C_{12}$ does not depend on the subgraph neither on the induced domain.
\end{rem}

We can rearrange the result of Theorem \ref{thm : CEG} to get
\begin{align} \label{ineg : 1}
\sigma_k(N, \Sigma) \leq C_{12} \cdot \frac{|N|_g^{\frac{d-2}{d}}}{|\Sigma|} \cdot k^{\frac{2}{d}},
\end{align}
which is more adequate for the purpose of this proof.

\subsection{Isoperimetric control of a domain associated to a subgraph}
This subsection is devoted to state an isoperimetric inequality satisfied by a domain $(N, \Sigma)$ such as explained in the previous one.
\begin{prop}
Let $(\Omega, B)$ be a subgraph of $\Gamma$, a Cayley graph with polynomial growth rate of order $d \geq 2$. Let $M$ be a modeled manifold and $(N, \Sigma)$ be the domain of $M$ associated to $(\Omega, B)$. Then there exists a constant $C_{14}$ depending only on $\Gamma$ and $P$ such that we have
\begin{align} \label{ineg : 2}
|N|_g^{\frac{d-2}{d}} \leq C_{14} \cdot |\Sigma|^{\frac{d-2}{d-1}}.
\end{align}
\end{prop}
\begin{proof}
As stated by Proposition \ref{prop : isoperimetrie}, we have 
\begin{align*}
\frac{|\bar{\Omega}|^{\frac{d-1}{d}}}{|B|} \leq C_2.
\end{align*}

By construction of the domain $N$, each vertex $i \in \bar{\Omega}$ adds a piece $P_i$ 
to $N$, so we have
\begin{align*}
|N|_g \leq |\bar{\Omega}| \cdot |P|_g,
\end{align*}
and each vertex $j \in B$ adds at least one copy of $\mathbb{S}^{d-1}$ to the boundary so we have
\begin{align} \label{ineg : sigma B}
|\Sigma| \geq |B| \cdot |\mathbb{S}^{d-1}|.
\end{align}
Altogether, this gives us 
\begin{align*}
\frac{|N|_g^{\frac{d-1}{d}}}{|\Sigma|} & \leq \frac{\left(|\bar{\Omega}| \cdot |P|_g\right)^{\frac{d-1}{d}}}{|B| \cdot |\mathbb{S}^{d-1}|} \\
            & \leq C_2 \cdot \frac{|P|_g^{\frac{d-1}{d}}}{|\mathbb{S}^{d-1}|} \\
            & = C_{13}.
\end{align*}
Raising it to the power $\frac{d-2}{d-1}$ gives us
\begin{align*}
\frac{|N|_g^{\frac{d-2}{d}}}{|\Sigma|^{\frac{d-2}{d-1}}} \leq  C_{13}^{\frac{d-2}{d-1}},
\end{align*}
which leads to 
\begin{align*} 
|N|_g^{\frac{d-2}{d}} \leq C_{14} \cdot |\Sigma|^{\frac{d-2}{d-1}}.
\end{align*}
\end{proof}

\subsection{Discretization of a Riemannian manifold}

Let us begin by explaining the strategy that motivates the next subsection. We have in our possession a subgraph $(\Omega, B)$ and a bounded domain $(N, \Sigma)$ associated to it. Moreover, we know how to estimate the Steklov spectrum of $N$, see (\ref{ineg : 1}). What we shall do from now is to associate a graph with boundary denoted $(\tilde{V}, \tilde{E}, V_\Sigma)$ to the domain $N$ in such a way that we will be able to estimate the Steklov spectrum of $(\tilde{V}, \tilde{E}, V_\Sigma)$ according to the one of $N$. This new graph with boundary $(\tilde{V}, \tilde{E}, V_\Sigma)$ is not our starting subgraph $(\Omega, B)$ but will be roughly isometric to it (see definition \ref{def : quasi isom}), which means that  $(\tilde{V}, \tilde{E}, V_\Sigma)$ and $(\Omega, B)$ are close enough for us to compare their spectra, which will allow us to conclude.
\begin{center}
\begin{tikzpicture}
   \node (1) at (0,0) {$\Gamma$} ;
   \node (2) at (0, -2) {$M$} ;
   \node (3) at (6, 0) {$(\Omega, B)$} ;
   \node (4) at (6, -2) {$(N, \Sigma)$} ;
   \node (5) at (10, -1) {$(\tilde{V}, \tilde{E}, V_\Sigma)$} ;
   
   \draw[->] (1) to node[pos=0.5]{\ modeled manifold} (2);
   \draw[->] (1) to node[pos=0.5, above]{subgraph} (3) ;
   \draw[->] (2) to node[pos=0.5, above]{domain} (4) ;
   \draw[<->] (3) to node[pos=0.5]{\: \ struture preserving} (4) ;
   \draw[->] (4) to node[pos=0.4, below right]{discretization} (5) ;
   \draw[<->] (3) to node[pos=0.4, above right]{roughly isometric} (5) ;
\end{tikzpicture}
\end{center}

We have to explain how to discretize a manifold. For further investigations on this topic, see \cite{CGR}.
\begin{defn}
We denote by $\mathcal{M} = \mathcal{M}(\kappa, r_0, d)$ the class of all compact Riemannian manifolds $N$ of dimension $d$ with smooth boundary $\Sigma$ satisfying the following assumptions:

There exist constants $\kappa \geq 0$ and $r_0 \in (0,1)$ such that
\begin{itemize}
\item The boundary  $\Sigma$ admits a neighborhood which is isometric to the cylinder $[0,1] \times \Sigma$, with the boundary corresponding to $\{0\} \times \Sigma$ ;
\item The Ricci curvature of $N$ is bounded below by $-(d-1)\kappa$ ;
\item The Ricci curvature of $\Sigma$ is bounded below by $-(d-2) \kappa $ ;
\item For each point $x \in N$ such that $d_M(x, \Sigma) >1$, $\mbox{inj}_M(x) > r_0$ ;
\item For each point $x \in \Sigma$, $\mbox{inj}_\Sigma(x) > r_0$.
\end{itemize}
\end{defn}

\begin{rem}
Because of the regularity of the modeled manifold $M$ and the compactness of the fundamental piece $P$, it is clear that there is some $\kappa$ and $r_0$ such that each bounded domain $N \subset M$ satisfies the last four assumptions. Moreover, a domain $N$ associated with a subgraph satisfies the first one as well, as state by the remark \ref{rem : vg isom}.
\end{rem}

\begin{defn} \label{def : quasi isom}
A rough isometry between two metric spaces $(X, d_X)$ and $(Y, d_Y)$ is a map $\phi : X \longrightarrow Y$ such that there exist constants $C_{15} >1, \; C_{16}, C_{17} >0$ satisfying 
\begin{align*}
C_{15}^{-1} \cdot d_X(x,y) -C_{16} \leq d_Y(\phi(x), \phi(y)) \leq C_{15} \cdot d_X(x,y) + C_{16}
\end{align*}
for all $x,y \in X$ and which satisfy 
\begin{align*}
\bigcup_{x \in X} B(\phi(x), C_{17}) = Y.
\end{align*}
\end{defn}

\begin{defn}
Given $\epsilon \in (0, \frac{r_0}{4})$, an $\epsilon$-discretization of a manifold $N \in \mathcal{M}$ is a procedure allowing to associate a graph with boundary $(\tilde{V}, \tilde{E}, V_\Sigma)$ to $N$, such that $N$ is roughly isometric to $(\tilde{V}, \tilde{E}, V_\Sigma)$.
\end{defn}

We now explain the procedure of discretization. Given $\epsilon  \in (0, \frac{r_0}{4})$, let $V_\Sigma$ be a maximal $\epsilon$-separated set in $\Sigma$. Let $V_\Sigma'$ be a copy of $V_\Sigma$ located $4 \epsilon$ away from $\Sigma$, 
\begin{align*}
V_\Sigma' = \{4\epsilon \} \times V_\Sigma \subset N.
\end{align*}
Let $V_I$ be a maximal $\epsilon$-separated set in $N \backslash [0, 4\epsilon) \times \Sigma$ such that $V_\Sigma' \subset V_I$.
The set $\tilde{V} = V_\Sigma \cup V_I$ is endowed with a structure of a graph, declaring
\begin{itemize}
\item Any two $v_1, v_2 \in \tilde{V}$ are adjacent if $d_M(v_1,v_2) \leq 3\epsilon$ ;
\item Any $v \in V_ \Sigma$ is adjacent to $v'=(4\epsilon, v) \in V_\Sigma'$.
\end{itemize}
The graph $(\tilde{V}, \tilde{E})$ obtained is a graph with boundary $(\tilde{V},\tilde{E},V_\Sigma)$, declaring $V_\Sigma$ as the boundary, $V_I$ as the interior. We shall call it $(\tilde{V}, V_\Sigma)$.

\begin{thm}
Given $\epsilon \in (0, \frac{r_0}{4})$, there exist constants $C_{18}, C_{19} >0$ depending on $\kappa, r_0, d$ and $\epsilon$ such that any $\epsilon$-discretization $(\tilde{V}, V_\Sigma)$ of a manifold $N \in \mathcal{M}(\kappa, r_0, d)$ satisfies
\begin{align*}
\frac{C_{18}}{k} \leq \frac{\sigma_k(N, \Sigma)}{\sigma_k(\tilde{V}, V_\Sigma)} \leq C_{19},
\end{align*}
for each $k \leq |V_\Sigma|$.
\end{thm}
This Theorem is exactly the Theorem $3$ point $4)$ of \cite{CGR}, one can look at for a proof.
\medskip

As an immediate consequence we have
\begin{align} \label{ineg : 3}
\sigma_k(\tilde{V}, V_\Sigma) \leq \frac{\sigma_k(N,\Sigma) \cdot k}{C_{18}},
\end{align}
which is more useful for us.

\begin{defn}
A rough  isometry between two graphs with boundary $(\Gamma_1, B_1)$ and $(\Gamma_2, B_2)$ is a rough isometry that sends $B_1$ to $B_2$.
\end{defn}

Now we shall emphasize the link between a subgraph $(\Omega, B)$ and an $\epsilon$-discretization $(\tilde{V}, V_\Sigma)$ of $N$, which is the purpose of the following Proposition.

\begin{prop}
Let $\Gamma = Cay(G,S)$ be a Cayley graph. Let $(\Omega, B)$ be a subgraph of $\Gamma$, $M$ be a manifold modeled on $\Gamma$ and $(N, \Sigma) \subset M$ be the bounded domain of $M$ associated to $(\Omega, B)$ as before. Let $(\tilde{V}, V_\Sigma)$ be any $\epsilon$-dicretization of $N$.

Then there exist constants $C_{15} >1, C_{16}, C_{17} >0$ depending only on $\Gamma$, $M$ and $\epsilon$ such that there exists a rough isometry $\phi : (\tilde{V}, V_\Sigma) \longrightarrow (\Omega, B)$ with constants $C_{15}, C_{16}, C_{17}$.
\end{prop}

\begin{rem}
The essential point of this Proposition is to state that the constants of the rough isometry can be chosen independently of the subgraph $(\Omega, B)$.
\end{rem}

\begin{proof}
Define 
\begin{align*}
\phi : (\tilde{V}, V_\Sigma) & \longrightarrow (\Omega, B) \\
\end{align*}
by :
\begin{itemize}
\item If $v \in V_\Sigma$, then $v$ is such that $v \in P_j$ for $j \in B$ and we define $\phi(v) =j$ ;
\item If $v \in V_I$ is such that $v \in P_i$ for $i \in \Omega$, we define $\phi(v) = i$ ;
\item If $v \in V_I$ is such that $v \in P_j$ for $j \in B$, we define $\phi(v)=i$ such that $i \in \Omega$ and $i \sim j$. If there are many such $i$, one amoung the at most $|S|$ possibilities is chosen once and for all ;
\item If $v$ lies on the gluing of two pieces,  one of the two possibilities is chosen once and for all.
\end{itemize}

Define $C_{15}$ as the triple of the cardinal of the biggest set of points $\epsilon$-separated of $P$. By compactness of $P$, $C_{15}$ is finite. From this definition it is forward that for $v_1, v_2 \in V$ such that $v_1$ belongs to the same piece as $v_2$, we have $d_{\tilde{V}}(v_1, v_2) \leq C_{15}$.

Recall that we chose the domain $N$ such that $N$ gets the same neighbor structure than the subgraph : for $i, j \in \bar{\Omega}$, we have $i \sim j \Leftrightarrow P_i \sim P_j$.

Now we define $C_{16} := C_{15}$. Hence, for all $v_1, v_2 \in \tilde{V}$,
\begin{align*}
d_{\tilde{V}}(v_1, v_2) \leq C_{15} \cdot d_{\bar{\Omega}}(\phi(v_1), \phi(v_2)) + C_{16}.
\end{align*}
In the same way we also have 
\begin{align*}
C_{15}^{-1} \cdot d_{\bar{\Omega}}(\phi(v_1), \phi(v_2)) - C_{16} \leq d_{\tilde{V}}(v_1, v_2).
\end{align*}
Remark now that $\phi$ is a surjective map. Hence we have
\begin{align*}
\bigcup_{v \in \tilde{V}} B(\phi(v), C_{17}) = (\Omega, B)
\end{align*}
for any value of $C_{17} > 0$. We can choose $C_{17} = 1$.
\end{proof}

This link between $(\Omega, B)$ and $(\tilde{V}, V_\Sigma)$ shall be exploited to give a relationship between the Steklov eigenvalues of theses graphs with boundary. We state here the Proposition $16$ of \cite{CGR} :

\begin{prop}
Given  $C_{15} >1, C_{16}, C_{17} >0$ there exist constants $C_{20}, C_{21}$  depending only on $C_{15}, C_{16}, C_{17}$ and on the maximal degree of the vertices, such that any two graphs with boundary $(\Gamma_1, B_1)$ and $(\Gamma_2, B_2)$ which are roughly isometric with constants $C_{15}, C_{16}, C_{17}$, satisfy
\begin{align*}
C_{20} \leq \frac{\sigma_k(\Gamma_1, B_1)}{\sigma_k(\Gamma_2, B_2)} \leq C_{21}
\end{align*}
for all $k  < \min\{|B_1|, |B_2|\}$.
\end{prop} 

Applied to our graphs, this leads to the existence of constants $C_{20}, C_{21}$ such that 
\begin{align*}
C_{20} \leq \frac{\sigma_k(\Omega, B)}{\sigma_k(\tilde{V}, V_\Sigma)} \leq C_{21}
\end{align*}
for all $k < |B|$, which we can rearrange in
\begin{align} \label{ineg : 4}
\sigma_k(\Omega, B) \leq \sigma_k(\tilde{V}, V_\Sigma) \cdot C_{21}.
\end{align}


Let us conclude our proof of Theorem \ref{thm : principal} by assembling the different results we obtained before.
\begin{align*}
\sigma_k(\Omega, B) & \stackrel{(\ref{ineg : 4})}{\leq} \sigma_k(\tilde{V}, V_\Sigma) \cdot C_{21} \\
                    & \stackrel{(\ref{ineg : 3})}{\leq} \frac{\sigma_k(N,\Sigma) \cdot k}{C_{18}} \cdot C_{21} \\
                    & \stackrel{(\ref{ineg : 1})}{\leq}  \frac{C_{12} \cdot \frac{|N|^{\frac{d-2}{d}}}{|\Sigma|} \cdot k^{\frac{2}{d}} \cdot k}{C_{18}} \cdot C_{21} \\
                    & \stackrel{(\ref{ineg : 2})}{\leq}  \frac{C_{12} \cdot \frac{ C_{14} \cdot |\Sigma|^{\frac{d-2}{d-1}}}{|\Sigma|} \cdot k^{\frac{2}{d}} \cdot k}{C_{18}} \cdot C_{21} \\ 
                    & := C_{22} \cdot \frac{1}{|\Sigma|^{\frac{1}{d-1}}} \cdot k^{\frac{d+2}{d}} \\                   
                    & \stackrel{(\ref{ineg : sigma B})}{\leq} C_{22} \cdot \frac{1}{(|B| \cdot |\mathbb{S}^{d-1}|)^{\frac{1}{d-1}}} \cdot k^{\frac{d+2}{d}} \\
                    & := C_{23} \cdot \frac{1}{|B|^{\frac{1}{d-1}}} \cdot k^{\frac{d+2}{d}}.
\end{align*}

Throughout the paper, we took care to specify on which parameters the constants depend. It happens that they depend only on $\Gamma, P$ and $\epsilon$, not on the subgraph $(\Omega, B)$ or the domain $(N, \Sigma)$ associated. Hence, if we set a fundamental piece $P$ associated to $\Gamma$, and if we set a value of $\epsilon$, the constant $C_{23}$ is now fixed.

Then, given $\Gamma$ a Cayley graph with polynomial growth rate of order $d \geq 2$, one can find a constant $C_{23} := C(\Gamma)$ such that for any subgraph $(\Omega, B)$ of $\Gamma$, we have 
\begin{align*}
\sigma_k(\Omega, B)  \leq C(\Gamma)  \cdot \frac{1}{|B|^{\frac{1}{d-1}}} \cdot k^{\frac{d+2}{d}},
\end{align*}
which proves Theorem \ref{thm : principal}.
\medskip

From this statement, let us prove Corollary \ref{corollaire : sigma_k}.
\medskip

Let $\Gamma = Cay(G,S)$ and $C(\Gamma)$ as above. Let $(\Omega_l, B_l)_{l=1}^\infty$ be a family of subgraphs such that $|\Omega_l| \underset{l \to \infty}{\longrightarrow} \infty$.

Because of the isoperimetric control stated by Proposition \ref{prop : isoperimetrie}, it is clear that $|B_l|  \underset{l \to \infty}{\longrightarrow} \infty$ too.
\medskip

Hence, for all $k$ fixed, we have
\begin{align*}
\sigma_k(\Omega_l, B_l) \leq C(\Gamma) \cdot \frac{1}{|B_l|^{\frac{1}{d-1}}} \cdot k^{\frac{d+2}{d}} \underset{l \to \infty}{\longrightarrow} 0,
\end{align*} 

which proves Corollary \ref{corollaire : sigma_k}.

\printbibliography

Université de Neuchâtel, Institut de Mathématiques, Rue Emile-Argand 11, CH-2000 Neuchâtel, Switzerland

E-mail address : leonard.tschanz@unine.ch

\end{document}